\documentclass{amsart}
\usepackage[dvips]{graphicx}
\usepackage{enumerate}
\usepackage{amssymb}
\usepackage{amscd}
\usepackage{amsmath}
\usepackage{color}
\usepackage{here}
\usepackage{mathrsfs}

\usepackage{lineno}

\begin{document}

\title[Bounded cohomology and optimal separating constant]{Bounded cohomology and optimal separating constant for representations}

\author{Yushi Nakano}
\address[Yushi Nakano]{Department of Mathematics, Hokkaido University, Kita 10, Nishi 8, Kita-Ku, Sapporo, Hokkaido, 060-0810, Japan}
\email{yushi.nakano@math.sci.hokudai.ac.jp}

\author{Teruhiko Soma}
\address[Teruhiko Soma]{Department of Mathematical Sciences, Tokyo Metropolitan University, 
Minami-Ohsawa 1-1, Hachioji, Tokyo 192-0397, Japan}
\email{tsoma@tmu.ac.jp}

\begin{abstract}
Let $\Sigma$ be a closed oriented surface of genus $>1$ and $M$ 
a complete hyperbolic 3-manifold with a marking $i:\Sigma\longrightarrow M$.
We consider the case that $M$ has no parabolic cusps and at least one of the two ends 
is simply degenerate.
For $\varGamma=\pi_1(\Sigma)$, let $\rho_M:\varGamma\longrightarrow \mathrm{PSL}_2(\mathbb{C})$ be the holonomy of $M$ and 
$\rho:\varGamma\longrightarrow \mathrm{PSL}_2(\mathbb{C})$ any representation in $\mathrm{PSL}_2(\mathbb{C})$.
We will show that, if $\rho$ is discrete and non-faithful, then 
\[
\|[\mathrm{Vol}(\rho)]-[\mathrm{Vol}(\rho_M)]\|_\infty\geq \boldsymbol{v}_3
\]
holds, where 
$[\mathrm{Vol}(\rho)]$ denotes the bounded fundamental class of $\rho$ in the bounded cohomology $H_b^3(\varGamma,\mathbb{R})$ of $\varGamma$ and $\boldsymbol{v}_3$ is the volume of a regular ideal 3-simplex in $\mathbb{H}^3$.
As an application, we present a rigidity theorem for $\rho_M$ in the set of representations $\rho$ of $\varGamma$ in $\mathrm{PSL}_2(\mathbb{C})$ in terms of $[\mathrm{Vol}(\rho)]$. 
The rigidity theorem implies that $\boldsymbol{v}_3$ is the optimal separating constant.
\end{abstract}

\subjclass[2020]{57K32, 30F40}
\keywords{Hyperbolic 3-manifolds, Kleinian groups, bounded cohomology}
\thanks{}
\date{\today}

\maketitle

\newtheorem{theorem}{Theorem}[section]
\newtheorem{cor}[theorem]{Corollary}
\newtheorem{lemma}[theorem]{Lemma}
\newtheorem{prop}[theorem]{Proposition}

\newtheorem{mtheorem}{Theorem}
\renewcommand{\themtheorem}{\Alph{mtheorem}}
\newtheorem{mcorollary}[mtheorem]{Corollary}

\theoremstyle{definition}
\newtheorem{definition}[theorem]{Definition}
\newtheorem{example}[theorem]{Example}
\newtheorem{remark}[theorem]{Remark}
\newtheorem{claim}[theorem]{Claim}
\newtheorem*{question}{Question}

\numberwithin{figure}{section}
\numberwithin{equation}{section}

Throughout this paper, we denote by $\Sigma$ a closed oriented hyperbolic surface of 
genus $>1$ 
and by $M$ a complete oriented hyperbolic 3-manifold 
with a marking $i:\Sigma\longrightarrow M$.
We consider the 3-cocycle $\omega_M:C_3(M)\longrightarrow \mathbb{R}$ such that, 
for any singular 3-simplex $\sigma:\Delta^3\longrightarrow M$, $\omega_M(\sigma)$ is the oriented volume of the 3-simplex $\mathrm{straight}(\sigma)$ 
obtained by straightening $\sigma$.
It is a well-know fact in hyperbolic geometry that the supremum norm $\|\omega_M\|_\infty$ of $\omega_M$ is equal to the volume of a regular ideal simplex $\boldsymbol{v}_3=1.01494\dots$ in $\mathbb{H}^3$.
So $\omega_M$ represents the bounded fundamental $[\omega_M]$ of the bounded cohomology $H_b^3(M,\mathbb{R})$ with $\|[\omega_M]\|_\infty\leq \boldsymbol{v}_3$.

In \cite[Theorem A]{so_BC}, \cite[Section 6]{om}, \cite[Corollary 1.5]{fa_gafa}, 
\cite[Theorems A and C]{so_tran} and other works, 
rigidity theorems for completed hyperbolic 3-manifolds 
$M_j$ $(j=0,1)$ with markings $i_j:\Sigma\longrightarrow M_j$ are proved under suitable situations.
In fact,  they presented a 
separating constant (or a threshold) $c>0$ such that, if 
\[
\|i_0^*[\omega_{M_0}]-i_1^*[\omega_{M_1}]\|_\infty <c
\]
in $H_b^3(\Sigma,\mathbb{R})$, 
then there exists a marking-preserving bi-Lipschitz map $\varphi:M_0\longrightarrow M_1$.
Moreover in \cite{so_tran} it is also shown that $\boldsymbol{v}_3$ is the optimal separating 
constant.

Let $\varGamma$ be the fundamental group of $\Sigma$ and $\mathcal{R}(\varGamma)$ the set of representations $\rho:\varGamma\longrightarrow \mathrm{PSL}_2(\mathbb{C})$.
The holonomy $\rho_M$ of $M$ is a discrete and faithful element of $\mathcal{R}(\varGamma)$.
Farre \cite{fa_imrn} defined the bounded fundamental class $[\mathrm{Vol}(\rho)]$ of $\rho$ in $H_b^3(\varGamma,\mathbb{R})$ such that 
$[\mathrm{Vol}(\rho_M)]$ is equal to $i^* [\omega_M]$ under the natural identification  $H_b^3(\varGamma,\mathbb{R})= H_b^3(\Sigma,\mathbb{R})$.
In \cite[Theorem 3.12]{fa_imrn} and \cite[Theorem D]{so_tran}, rigidity theorems 
for $\rho_M$ in $\mathcal{R}(\varGamma)$ are proved under suitable 
conditions and with some separating constants.
In the proofs of them, Farre's separation theorem (\cite[Theorem 4.11]{fa_imrn}) 
plays an important role, which shows that,  if $\rho(\varGamma)$ is dense in $\mathrm{PSL}_2(\mathbb{C})$, then the following inequality
\begin{equation}\label{eqn_geq_v3}
\|[\mathrm{Vol}(\rho)]-[\mathrm{Vol}(\rho_M)]\|_\infty\geq \boldsymbol{v}_3
\end{equation}
holds.

Theorem \ref{thm_A} shows that \eqref{eqn_geq_v3} 
still holds in some cases even if  $\rho(\varGamma)$ is not dense in $\mathrm{PSL}_2(\mathbb{C})$.

\begin{mtheorem}\label{thm_A}
Suppose that $M$ has no parabolic cusps and at least one of the two ends of $M$ 
is simply degenerate.
Then, for any discrete and non-faithful element $\rho$ of $\mathcal{R}(\varGamma)$  
(possibly $\rho(\varGamma)$ contains parabolic or elliptic elements), the inequality 
\eqref{eqn_geq_v3} holds.
\end{mtheorem}

The no parabolic-cusp condition on $M$ is crucial in the proof of Lemma \ref{l_Lemma5.1}. 

In the case when $\rho$ is discrete and faithful, the quotient space $M_\rho=\mathbb{H}^3/\rho(\varGamma)$ is a complete hyperbolic 3-manifold which has 
 a fixed marking $i_\rho:\Sigma\longrightarrow M_\rho$ with $\rho_{M_\rho}=\rho$ up to conjugacy in $\mathrm{PSL}_2(\mathbb{C})$.

\medskip

Theorem \ref{thm_A} together with  \cite[Theorem A]{so_tran} 
shows the following rigidity theorem for $\rho_M$ in $\mathcal{R}(\varGamma)$.

\begin{mcorollary}[cf.\ {\cite[Theorem D]{so_tran}}]\label{thm_B}
Suppose that $M$ has no parabolic cusps and at least one of the two ends of $M$ 
is simply degenerate, which is denoted by $e$.
If $\rho$ is any element of $\mathcal{R}(\varGamma)$ satisfying  
\begin{equation}\label{eqn_thm_D}
\|[\mathrm{Vol}(\rho)]-[\mathrm{Vol}(\rho_M)]\|_\infty<\boldsymbol{v}_3,
\end{equation}
then the following {\rm (i)} and {\rm (ii)} hold.
\begin{enumerate}[\rm(i)]
\item
$\rho$ is discrete and faithful.
\item
There exists a marking-preserving homeomorphism $\varphi:M\longrightarrow M_\rho$ 
such that $\varphi|_E:E\longrightarrow \varphi(E)$ is bi-Lipschitz 
for some neighborhood $E$ of $e$ in $M$.
\end{enumerate}
\end{mcorollary}

As is noted in \cite[Remark 0.1]{so_tran}, the separation constant $\boldsymbol{v}_3$ of \eqref{eqn_thm_D} is optimal.
In fact, if $\rho_i:\varGamma\longrightarrow \mathrm{PSL}_2(\mathbb{C})$ $(i=0,1)$ are discrete and faithful representations 
without parabolic elements and such that at least one end of $M_{\rho_0}$ is simply degenerate but the both ends of $M_{\rho_1}$ are 
geometrically finite, then \cite[Theorem 1]{so_duke} implies that 
\[
\|[\mathrm{Vol}(\rho_0)]-[\mathrm{Vol}(\rho_1)]\|_\infty=\|[\mathrm{Vol}(\rho_0)]\|_\infty= \boldsymbol{v}_3.
\]
However any marking-preserving homeomorphism $\varphi:M_{\rho_0}\longrightarrow M_{\rho_1}$ 
does not satisfy the condition (ii) of Corollary  \ref{thm_B} 
since any simply degenerate end is not bi-Lipschitz to a geometrically finite end.

\medskip

The following corollary follows immediately from Corollary \ref{thm_B} together with Sullivan's rigidity theorem \cite{su}.

\begin{mcorollary}\label{cor_C}
With the conditions as in Corollary \ref{thm_B}, 
suppose furthermore that the both ends of $M$ are simply degenerate.
Then $\rho$ is discrete, faithful and there exists a marking-preserving isometry $\varphi:M\longrightarrow M_\rho$.
\end{mcorollary}

\section{Preliminaries}

In this section, we present fundamental definitions and notations in forms suitable
to our arguments. Refer to Thurston \cite{th}, Benedetti and Petronio \cite{bp},
Matsuzaki and Taniguchi \cite{mt} and so on for other notations concerning hyperbolic
geometry and to Hempel \cite{he} for 3-manifold topology.

\subsection{Ends of hyperbolic 3-manifolds}

Let $N$ be a complete hyperbolic 3-manifold with finitely generated fundamental 
group and $N_{\mathrm{cusp}}$ the union of parabolic cusps of $N$ with respect to 
a fixed Margulis constant.
Then $M_{\mathrm{main}}=M\setminus \mathrm{Int}\,M_{\mathrm{cusp}}$ is the \emph{main part} of $M$.
By Scott's core theorem \cite{sc}, there exists a compact irreducible 3-dimensional 
submanifold $C$ of $N$ such that the inclusion $C\longrightarrow N$ is a homotopy 
equivalence.
Then $C$ is called a \emph{compact core} of $N$.
By the topological tameness theorem \cite{ag,cg} for hyperbolic 3-manifolds, 
one can choose the core $C$ so that the the closure $E_C$ of each component $N\setminus C$ 
is homeomorphic to $S_C\times [0,\infty)$, where $S_C$ is the frontier of $E_C$ in $N$.
See Figure \ref{f_core}.
\begin{figure}[hbtp]
\centering
\scalebox{0.6}{\includegraphics[clip]{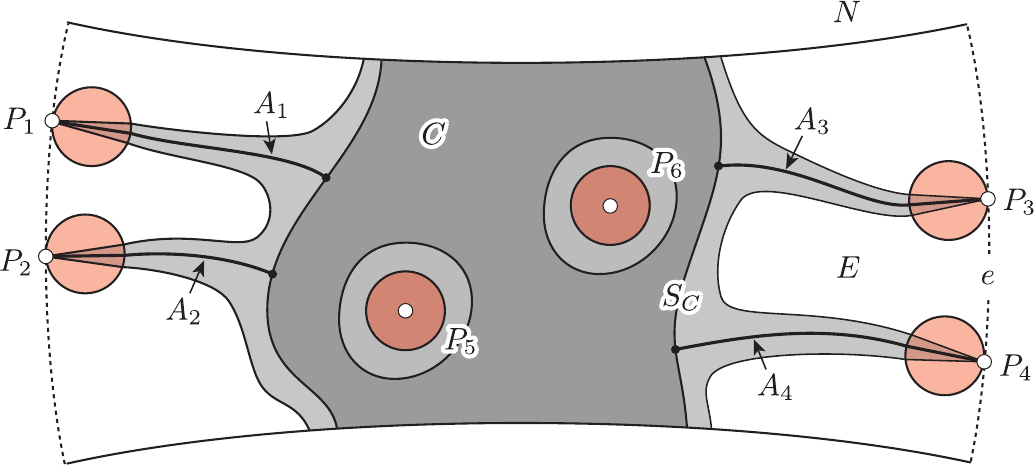}}
\caption{$P_1,\dots,P_4$ represent $\mathbb{Z}$-cusps and $P_5$, $P_6$ $\mathbb{Z}\times\mathbb{Z}$-cusps of $N$.
$S_C$ is the frontier of a component $E_C$ of $N\setminus \mathrm{Int}\,C$.
The union of dark regions is a finite core $\widehat C$.
$E$ is the neighborhood of $e$ with respect to $\widehat C$.
}
\label{f_core}
\end{figure}
Let $P_1,\dots,P_n$ be the $\mathbb{Z}$-cusps of $N$.
Then there exist mutually disjoint half-open annuli $A_1,\dots, A_n$ properly embedded 
in $N\setminus \mathrm{Int}\,C$ such that $A_i\cap P_i$ $(i=1,\dots,n)$ is a totally geodesic cusp and 
$A_i\cap P_j=\emptyset$ for any $j\neq i$.
One can take a regular neighborhood  $C_{\mathcal{A}}$ of $C\cup(\bigcup_{i=1}^n A_i)$ in $N$ of finite volume.
The union $\widehat C$ of $C_{\mathcal{A}}$ and the components of $N\setminus \mathrm{Int}\,C$ containing $\mathbb{Z}\times \mathbb{Z}$-cusps of $N$ is called a \emph{finite core} of $N$.
Let $e$ be an end of $N_{\mathrm{main}}$.
The component $E$ of $N\setminus \mathrm{Int}\,\widehat C$ adjacent to 
$e$ is called the neighborhood of $e$ \emph{with respect to} $\widehat C$.
We say that $e$ is \emph{geometrically finite} if there exists a neighborhood of 
$e$ in $N_{\mathrm{main}}$ disjoint from the convex hull of $N$ and otherwise 
\emph{geometrically infinite}.
By the geometrical tameness theorem \cite{ca93} for hyperbolic 3-manifolds, 
any geometrically infinite end $e$ of $N_{\mathrm{main}}$ is simply degenerate, that is, 
there exists a sequence of simplicial hyperbolic surfaces tending toward $e$.
From now on ends of $N_{\mathrm{main}}$ will simply be called ends of $N$.

\subsection{Bounded cohomology of spaces}
We present some notions concerning bounded cohomology of topological spaces.
See \cite{gr} for details.

Let $C^*(X)$ be the dual space of the singular chain-complex $C_*(X)$ of a topological space $X$ with 
real coefficient.
Consider the subspace $C_b^*(X)$ of $C^*(X)$ consisting of bounded cochains, that is, 
$c\in C_b^n(X)$ means that 
$$\|c\|_\infty=\sup\left\{\,|c(\sigma)|\ ;\ \text{$\sigma:\Delta^n\longrightarrow X$ is a singular $n$-simplex}\right\}<\infty,$$
where $\Delta^n$ is a regular $n$-simplex of edge length 1 in the Euclidean $n$-space.
Since the coboundary operator $\delta^n:C^n(X)\longrightarrow C^{n+1}(X)$ satisfies $\delta^n(C_b^n(X))\subset C_b^{n+1}(X)$, 
the bounded cochain complex $(C_b^*(X),\delta^*_b)$ with $\delta^*_b=\delta^*|_{C_b^*(X)}$ defines the \emph{bounded cohomology} $H_b^n(X,\,\mathbb{R})$
with the seminorm
$$\|\alpha\|_\infty=\inf\left\{\|c\|_\infty\ ;\, \text{$c$ is an element of $Z_b^n(X)$ with $[c]=\alpha$}\right\}$$
for $\alpha\in H_b^n(X,\,\mathbb{R})$, where 
$Z_b^n(X)=(\delta_b^n)^{-1}(0)$.

\medskip

Suppose that $N$ is a complete hyperbolic 3-manifold and 
$p:\mathbb{H}^3\longrightarrow N$ is the locally isometric universal covering.
For any singular $k$-simplex $\sigma:\Delta^k\longrightarrow N$, consider its lift 
$\widetilde \sigma:\Delta^k\longrightarrow \mathbb{H}^3$.
Let $\mathrm{straight}(\widetilde \sigma):\Delta^k\longrightarrow \mathbb{H}^3$ the affine map with respect to the Euclidean structure on $\Delta^3$ and the quadratic model on $\mathbb{H}^3$ 
with $\mathrm{straight}(\widetilde\sigma(v_j))=\widetilde\sigma(v_j)$ for all vertices $v_j$ 
$(j=0,1,\dots,k)$ of $\Delta^k$.
Then 
the map $\mathrm{straight}_{N}(\sigma)=p\circ \mathrm{straight}(\widetilde\sigma):\Delta^k\longrightarrow N$ is called the $k$-simplex obtained by \textit{straightening} $\sigma$.

The (oriented) volume of a $C^1$ singular 3-simplex $\sigma:\Delta^3\longrightarrow N$ is defined by 
\begin{equation}\label{eqn_omega_N}
\mathrm{Vol}(\sigma)=\int_{\Delta^3}\sigma^*(\Omega_N),
\end{equation}
where $\O_N$ is the volume form on $N$.
Let $\omega_N$ be the 3-cocycle on $N$ 
defined by
$$
\omega_N(\sigma)=\mathrm{Vol}(\mathrm{straight}_{N}(\sigma))
$$
for any singular 3-simplex $\sigma:\Delta^3\longrightarrow N$.
Let $\boldsymbol{v}_3$ be the volume $\boldsymbol{v}_3$ of a regular ideal 3-simplex in $\mathbb{H}^3$.
Since $|\omega_N(\sigma)|$ is less than $\boldsymbol{v}_3$  
for any singular 3-simplex $\sigma:\Delta^3\longrightarrow N$, 
$\omega_N$ represents an element $[\omega_N]$ of $H_b^3(N,\,\mathbb{R})$ with 
\begin{equation}\label{eqn_v3}
\|[\omega_N]\|_\infty\leq \boldsymbol{v}_3.
\end{equation}
We say that $[\omega_N]$ is the \emph{bounded fundamental class} of $N$.

\subsection{Bounded cohomology of groups}\label{ss_BCgroup}

We briefly review the bounded cohomology of groups and the bounded fundamental class of 
representations in $\mathrm{PSL}_2(\mathbb{C})$. 
See Section 2 of Farre \cite{fa_imrn} for details.

For a discrete group $G$ and $n\geq 0$, let 
$\ell^\infty(G^{n+1})$ be the $\mathbb{R}$-vector space of functions $f:G^{n+1}\longrightarrow \mathbb{R}$ 
with 
\[
\| f\|_\infty=\sup\{\,|f(g_0,\dots,g_n)|\,;\, (g_0,\dots,g_n)\in G^{n+1}\}<\infty.
\]
The sequence $\bigl(\ell^\infty(G^{n+1})\bigr)_{n=0}^\infty$ has 
a homogeneous coboundary 
$\delta^n:\ell^\infty(G^{n+1})\longrightarrow \ell^\infty(G^{n+2})$ defined by 
\[
\delta^n f(g_0,\dots,g_{n+1})=\sum_{i=0}^{n+1}(-1)^i f(g_0,\dots,\widehat g_i,\dots, g_{n+1}), 
\]
where $\widehat g_i$ means that $g_i$ is omitted.
Consider the left-action of $G$ on $\ell^\infty(G^{n+1})$ with  
\[
g\cdot f(g_1,\dots,g_n)=f(g^{-1}g_1,\dots,g^{-1}g_n)
\]
for $g\in G$ and $f\in \ell^\infty(G^{n+1})$.
For a subgroup $\varLambda$ of $G$, 
we say that $f$ is $\varLambda$-\emph{invariant} if $g\cdot f=f$ for any $g\in \varLambda$ and 
denote by $\ell^\infty(G^{n+1})^\varLambda$ the subspace of $\ell^\infty(G^{n+1})$ consisting of 
$\varLambda$-invariant elements.
From the definition,
\[
\ell^\infty(G^{n+1})^G\subset \ell^\infty(G^{n+1})^\varLambda\subset \ell^\infty(G^{n+1}).
\]
Since $\delta^n(\ell^\infty(G^{n+1})^G)\subset \ell^\infty(G^{n+2})^G$, 
the complex $(\ell^\infty(G^{n+1}), \delta^n|_{\ell^\infty(G^{n+1})^G})_{n\geq 0}$ defines the cohomology
\[
H_b^*(G,\mathbb{R})=H^*(\ell^\infty(G^{*+1})^G),
\]
which is called the \emph{bounded cohomology} of $G$. 
Note that $H_b^n(G,\mathbb{R})$ has the seminorm $\|\cdot\|_\infty$ with  
\[
\|\alpha\|_\infty=\inf\{\,\|f\|_\infty\,;\, \text{$f$ is an element of $Z_b^n(G)$ with $[f]=\alpha$}\}
\]
for $\alpha\in H_b^n(G,\mathbb{R})$, where $Z_b^n(G)=(\delta^n|_{\ell^\infty(G^{n+1})^G})^{-1}(0)$.

\medskip

Let $\rho:G\longrightarrow \mathrm{PSL}_2(\mathbb{C})$ be any representation of $G$.
Then we will define the bounded fundamental class $[\mathrm{Vol}(\rho)]$ of $\rho$ in $H_b^3(G,\mathbb{R})$.

Fix $x\in \mathbb{H}^3\cup \partial \mathbb{H}^3$.
We denote by $\mathrm{Vol}_x(\rho)$ the element of $\ell^\infty(G^{4})^G$ 
such that $\mathrm{Vol}_x(\rho)(g_0,\dots,g_3)$ is the oriented volume of the (ideal) straight 3-simplex in $\mathbb{H}^3$ with the oriented vertices $\rho(g_0)x,\dots,\rho(g_3)x$, 
where $\rho(g_i)$ is supposed to be an element of $\mathrm{Isom}^+(\mathbb{H}^3)$ under the natural identification $\mathrm{PSL}_2(\mathbb{C})=\mathrm{Isom}^+(\mathbb{H}^3)$.
Here we define that $\mathrm{Vol}_x(\rho)(g_0,\dots,g_3)=0$ if the 3-simplex is degenerate.
Since $\delta^3 \mathrm{Vol}_x(\rho)=0$ and $\|\mathrm{Vol}_x(\rho)\|_\infty\leq \boldsymbol{v}_3$, 
$\mathrm{Vol}_x(\rho)$ is a 3-cocyle of $\ell^\infty(G^{4})^G$ and hence it represents 
an element $[\mathrm{Vol}_x(\rho)]$ of $H_b^3(G,\mathbb{R})$ 
with $\|[\mathrm{Vol}_x(\rho)]\|_\infty\leq \boldsymbol{v}_3$.

As is shown in \cite[Subsection 2.4]{fa_imrn}, $[\mathrm{Vol}_x(\rho)]=[\mathrm{Vol}_y(\rho)]$ in 
$H_b^3(G,\mathbb{R})$ for any $x,y\in \mathbb{H}^3\cup \partial\mathbb{H}^3$.
So we may set $[\mathrm{Vol}(\rho)]=[\mathrm{Vol}_x(\rho)]$ and call it the \emph{bounded fundamental class} of $\rho$.
Since $\mathrm{Vol}_{gx}(g\rho g^{-1})=\mathrm{Vol}_x(\rho)$, it follows that $[\mathrm{Vol}(g\rho g^{-1})]=[\mathrm{Vol}(\rho)]$.
This means that $[\mathrm{Vol}(\rho)]$ is invariant 
under $\mathrm{PSL}_2(\mathbb{C})$-conjugacy.
For a subgroup $\varLambda$ of $G$, let 
$i_{\varLambda,G}:\varLambda\longrightarrow G$ be the inclusion and 
$\rho|_\varLambda:\varLambda\to \mathrm{PSL}_2(\mathbb{C})$ the restriction of $\rho$ on $\varLambda$, 
that is, $\rho|_\varLambda= \rho\circ i_{\varLambda,G}$.
It is a standard fact that the homomorphism $i_{\varLambda,G}^*:H_b^3(G,\mathbb{R})\to H_b^3(\varLambda,\mathbb{R})$ 
induced by $i_{\varLambda,G}$ is seminorm non-increasing.
So we have 
\[
\|[\mathrm{Vol}(\rho|_\varLambda)]\|_\infty=\|i_{\varLambda,G}^*[\mathrm{Vol}(\rho)]\|_\infty\leq \|[\mathrm{Vol}(\rho)]\|_\infty.
\]

Now we consider the case that $\rho:G\to \mathrm{PSL}_2(\mathbb{C})$ is discrete and faithful.
Then we know that, for the bounded fundamental class $[\omega_{M_\rho}]$ in $H_b^3(M_{\rho},\mathbb{R})$, 
$i_\rho^*[\omega_{M_\rho}]$ is equal to $[\mathrm{Vol}(\rho)]$ under the natural identification $H_b^3(\Sigma,\mathbb{R})=H_b^3(G,\mathbb{R})$, 
for example see \cite[Subsection 2.2]{fa_gafa}.

\section{Co-amenable restriction}

A subgroup $\varLambda$ of a group $L$ or the inclusion $i:\varLambda\longrightarrow L$ is said to be \emph{co-amenable} if there exists an $L$-invariant mean $m$ on the left coset $L/\varLambda$. 
More precisely, $m$ is a linear functional 
$
m:\ell^\infty(L/\varLambda)\longrightarrow\mathbb{R}
$
satisfying the following conditions.
\begin{itemize}
\item
$m(1)=1$ and $m(f)\ge0$ for any $f\in\ell^\infty(L/\varLambda)$ with $f\geq 0$.
\item
$m(g\cdot f)=m(f)$ 
for any $g\in L$ and $f\in\ell^\infty(L/\varLambda)$.
\end{itemize}
Here the action of $L$ on $\ell^\infty(L/\varLambda)$ is defined by
$(g\cdot f)(x\varLambda)=f(g^{-1}x\varLambda)$.
It is a standard fact that, for any $f\in \ell^\infty(L/\varLambda)$, 
\begin{equation}\label{eqn_mf_f}
|m(f)|\leq \|f\|_\infty.
\end{equation}
We say that $L$ is \emph{amenable} if the trivial group $\{1\}$ is 
a co-amenable subgroup of $L$.

Any finite group $G$ is amenable. 
Indeed, a $G$-invariant mean is explicitly given by 
$m(f) = \frac{1}{|G|} \sum_{g \in G} f(g)$ for $f:G\longrightarrow \mathbb{R}$.
The infinite cyclic group $\mathbb{Z}$ is also amenable.
Although a $\mathbb{Z}$-invariant mean cannot be written down explicitly, it 
can be rigorously defined and represented by employing a non-constructive limit concept known as a Banach limit.

\bigskip

The following lemma is a standard consequence of Gromov's amenable-resolution
construction for bounded cohomology \cite{gr}, 
which is refined by Monod \cite{mo}.

\begin{lemma}\label{l_Lemma3.1}
Let $\varLambda$ be a co-amenable subgroup of $L$. 
Then for any $n\ge0$ and $\alpha\in H_b^n(L;\mathbb R)$,
\begin{equation}\label{eqn_iHL}
\|i_{\varLambda,L}^*\alpha\|_\infty
=
\|\alpha\|_\infty.
\end{equation}
\end{lemma}
\begin{proof}
We may suppose that $L$ is a \emph{set} which admits the left actions of the 
groups $L$ and $\varLambda$. 
For any $x\in L$, the stabilizers $\mathrm{Stab}_L(x)$ and $\mathrm{Stab}_\varLambda(x)$ are 
trivial and hence in particular amenable.
This means that $L$ is an \emph{amenable $L$-set} and also an \emph{amenable $\varLambda$-set}.
Then Monod's amenable-set resolution theorem \cite{mo}  gives the canonical isometric identification of $H_b^*(\varLambda,\mathbb{R})=H^*(\ell^\infty(\varLambda^{*+1})^\varLambda)$ with $H^*(\ell^\infty(L^{*+1})^\varLambda)$.
See Definition 4.20, Lemmas 4.21, 4.22 and Theorem 4.23 in \cite{fr_17} 
for details.
So we can suppose that
\[
H_b^*(\varLambda,\mathbb{R})=H^*(\ell^\infty(L^{*+1})^\varLambda).
\]

Since $\varLambda$ is a co-amenable subset of $L$, there exists 
an $L$-invariant mean $m:\ell^{\infty}(L/\varLambda)\longrightarrow \mathbb{R}$. 
For $f \in \ell^\infty(L^{n+1})^\varLambda$ and $(g_0,\dots, g_n)\in L^{n+1}$,
the correspondence 
\begin{equation}\label{eqn_xHf}
x\varLambda\longmapsto f(x^{-1}g_{0}, \dots, x^{-1}g_{n})
\end{equation}
defines an element of $\ell^\infty(L/\varLambda)$.
Indeed, if we replace the representative $x$ of the coset
$x\varLambda$ by $x\lambda$ with $\lambda\in \varLambda$,
then
\[
\begin{aligned}
f((x\lambda)^{-1}g_0,\ldots,(x\lambda)^{-1}g_n)
&=
f(\lambda^{-1}x^{-1}g_0,\ldots,\lambda^{-1}x^{-1}g_n)\\
&=
f(x^{-1}g_0,\ldots,x^{-1}g_n),
\end{aligned}
\]
where the last equality follows from the $\varLambda$-invariance of $f$.
Therefore the map \eqref{eqn_xHf} 
is independent of the choice of representative of the coset.
So one can define the element $Tf$ of $\ell^\infty (L^{n+1})$ by 
\begin{equation}\label{eqn_Tf}
Tf(g_0,\dots,g_n)=m\bigl(x\varLambda\longmapsto f(x^{-1}g_{0}, \dots, x^{-1}g_{n})\bigr).
\end{equation}
Since moreover $m$ is $L$-invariant, $Tf$ is also $L$-invariant.
Thus \eqref{eqn_Tf} defines the averaging operator
\[
T:\ell^\infty(L^{n+1})^\varLambda\longrightarrow \ell^\infty(L^{n+1})^L.
\]
By \eqref{eqn_mf_f}, $\|Tf\|_\infty\le \|f\|_\infty$.
Since $T$ commutes with the homogeneous coboundary, it induces the 
homomorphism   
$T_*: H_b^n(\varLambda,\mathbb{R})\longrightarrow H_b^n(L,\mathbb{R})$ with 
$
\|T_*(\beta)\|_\infty\leq \|\beta\|_\infty
$
for any $\beta\in H_b^n(\varLambda,\mathbb{R})$.

If $f$ is an element of $\ell^\infty(L^{n+1})^L$, then $f(x^{-1}g_0,\dots, x^{-1}g_n)=f(g_0,\dots,g_n)$ for any $x\in L$.
Hence the function averaged in \eqref{eqn_Tf} is constant, and since moreover $m(1)=1$, we obtain $Tf=f$.
This shows that $T_*\circ i_{\varLambda,L}^*=\mathrm{id}_{H_b^n(L,\mathbb{R})}$ and hence 
\[
\|\alpha\|_\infty=\|T_*\circ i_{\varLambda,L}^*(\alpha)\|_\infty\leq \| i_{\varLambda,L}^*(\alpha)\|_\infty.
\]
for any $\alpha\in H_b^n(L,\mathbb{R})$.
On the other hand, since $i_{\varLambda,L}^*$ is seminorm non-increasing, 
\[
\| i_{\varLambda,L}^*(\alpha)\|_\infty\leq \|\alpha\|_\infty.
\]
The equality \eqref{eqn_iHL} is obtained immediately from these inequalities.
\end{proof}

\begin{lemma}\label{l_Lemma4.1}
Suppose that $\rho:\varGamma\longrightarrow \mathrm{PSL}_2(\mathbb{C})$ is a non-faithful representation such that 
$G=\rho(\varGamma)$ is non-trivial.
Then there exists a finitely generated free subgroup $A$ of $\varGamma$ with $\rho(A)=G$.
\end{lemma}

\begin{proof}
Since $\varGamma$ is a closed surface group of genus greater than one, we may assume that $\varGamma$ is a Fuchsian group such that the action of $\varGamma$ on $\mathbb{H}^2$ 
is faithful, proper and cocompact.
In particular, $\varGamma$ is \emph{of general type} in the sense of \cite[Definition 1.1]{gor}, that is, $\varGamma$ is non-elementary and does not fix any point of $\partial\mathbb{H}^2$.
By~\cite[Lemma~2.5]{gor}, the normal subgroup $K=\ker(\rho)$ either has a bounded orbit or is also of general type.
Since any non-trivial element of $K$ is hyperbolic, 
it does not have bounded orbits 
and hence $K$ is of general type.

Choose a finite generating set
$g_1,\ldots,g_r$
of $G$, and let
$x_1,\ldots,x_r$ be elements of $\varGamma$ with $\rho(x_i)=g_i$.
Applying Proposition~2.3 of~\cite{gor} to $K$, we obtain elements
$a_i,b_i\in K$
such that $y_i=b_ix_ia_i$ for $i\ge1$ form a free basis of a free subgroup of $\varGamma$.
Since $\rho(a_i)=\rho(b_i)=1$,
we have $\rho(y_i)=\rho(x_i)$ for every $i$.
Therefore 
$A=\langle y_1,\ldots,y_r\rangle$
is a free subgroup of rank $r$ satisfying $\rho(A)=G$.
\end{proof}

Recall that $\varGamma$ is a closed surface group of genus $>1$ and $M$ is a 
complete hyperbolic 3-manifold with holonomy $\rho_M:\varGamma\longrightarrow \mathrm{PSL}_2(\mathbb{C})$ 
and without parabolic cusps.

\begin{lemma}\label{l_Lemma5.1}
Let $A$ be the finitely generated free subgroup  of $\varGamma$ given in Lemma~\ref{l_Lemma4.1}.
Then 
$i_{A,\varGamma}^*[\mathrm{Vol}(\rho_M)]=0$ in $H_b^3(A,\mathbb{R})$.
\end{lemma}

\begin{proof}
Let 
$p:M_A\longrightarrow M$
be the covering corresponding to
$A$.
Suppose first that
$M_A$ has a simply degenerate end $e$ and introduce a contradiction.
By Canary's covering theorem \cite{ca}, 
$e$ has the 
neighborhood $\widetilde E$ with respect to a finite core of $M_A$ 
such that one of the following (i) and (ii) holds.
\begin{enumerate}[(i)]
\item
$ p|_{\widetilde E}:\widetilde E\longrightarrow M$ is a finite covering over an end neighborhood $E$ of $M$.
\item
$M$ has finite volume and has a finite cover which fibers over the circle.
\end{enumerate}
Since $M$ has a simply degenerate end, $M$ has infinite volume and hence 
alternative (ii) cannot occur.
Suppose that (i) did hold.
Then $(p|_{\widehat E})_*(\pi_1(\widetilde E))$ is a finite index subgroup of $\pi_1(E)$.
Since $M$ has no parabolic cusps, $M$ is homeomorphic to $S\times \mathbb{R}$ 
for a closed surface $S$ with $\pi_1(S)=\varGamma$, which has the two ends 
$S\times \{-\infty\}$ and $S\times \{\infty\}$.
Since $E$ is an end neighborhood in $M$, we may take $E$ so that it corresponds to either $S\times (-\infty,0]$ or $S\times [0,\infty)$ in $S\times \mathbb{R}$.
So the inclusion $i:E\longrightarrow M$ is a homotopy equivalence and 
hence $i_*(\pi_1(E))=\pi_1(M)$.
It follows that $p_*(\pi_1(\widetilde E))=i_*\circ (p|_E)_*(\pi_1(\widetilde E))$ is a finite index 
subgroup of $\pi_1(M)$.
Since $p_*(\pi_1(\widetilde E))$ is a finite-index subgroup of $\pi_1(M)=\varGamma$ and 
contained in $A$, the subgroup $A$ is also of finite index in $\varGamma$.
This contradicts that any free subgroup of a non-trivial closed surface group 
has infinite index.
Thus any end of $M_A$ is geometrically finite.
By \cite[Theorem 1]{so_duke}, we have $i_{A,\varGamma}^*[\mathrm{Vol}(\rho_M)]=0$.
\end{proof}

The following lemma is more or less folklore.

\begin{lemma}\label{l_Lemma6.1}
Let $G$ be a discrete torsion-free subgroup of $\mathrm{PSL}_2(\mathbb{C})$ such that $N = \mathbb{H}^3/G$ has finite volume.
For the inclusion $i_G:G\longrightarrow \mathrm{PSL}_2(\mathbb{C})$, 
\[
\|[\mathrm{Vol}(i_G)]\|_\infty = \boldsymbol{v}_3.
\]
\end{lemma}

\begin{proof}
By \eqref{eqn_v3}, $\|[\omega_N]\|_\infty \le \boldsymbol{v}_3$. 
Let $[N]$ denote the $\ell^1$-fundamental class of $N$.
From the definition \eqref{eqn_omega_N} of $\omega_N$, 
\begin{equation}\label{eqn_volN}
\mathrm{Vol}(N) = |\langle [\omega_N], [N] \rangle| \leq \|[\omega_N]\|_\infty \|[N]\|_1,
\end{equation}
where $\|[N]\|_1$ is the simplicial volume (Gromov invariant) of $N$, 
see \cite[Chapter 6]{th}.
The Gromov-Thurston proportionality formula (see Theorem 6.2 and 
Lemma 6.5.4 in \cite{th}) implies $\|[N]\|_1 = \mathrm{Vol}(N)/\boldsymbol{v}_3$.
By this fact together with \eqref{eqn_volN}, we have 
$\|[\omega_N]\|_\infty \ge \boldsymbol{v}_3$ and hence $\|[\mathrm{Vol}(i_G)]\|_\infty =\|[\omega_N]\|_\infty= \boldsymbol{v}_3$.
\end{proof}

\begin{lemma}\label{l_Proposition 7.1}
Suppose that $G$ is a finitely generated,
discrete, torsion-free subgroup of $\mathrm{PSL}_2(\mathbb{C})$.
Then one of the following {\rm (i)} and {\rm (ii)} holds.
\begin{enumerate}[\rm (i)]
\item
$\|[\mathrm{Vol}(i_G)]\|_\infty= 0$.
\item
There is a (not necessarily finitely generated) free subgroup $K$ of $G$ such that 
$\|\iota_{K,G}^*[\mathrm{Vol}(i_G)]\|_\infty=\boldsymbol{v}_3$.
\end{enumerate}
\end{lemma}

\begin{proof}
Write $N = \mathbb{H}^3/G$. 
If $G$ is elementary, then it is amenable.
Then its positive-degree bounded cohomology vanishes (for example see \cite[Corollary 7.5.12]{mo}) and  (i) holds. 
Assume henceforth that $G$ is non-elementary.

First we consider the case when $N$ is of finite volume (possibly $N$ has parabolic cusps).
By Agol's virtual fibering theorem (see \cite{ag08, hw, wa, ag13} and so on), there is a finite-index subgroup $G_1$ of $G$ admitting an exact sequence
\[
1\longrightarrow P\longrightarrow G_1\longrightarrow \mathbb{Z}\longrightarrow 1,
\]
where $P$ is the fundamental group of a fiber surface. 
Since $G_1$ has finite index in $G$, the inclusion
$i_{G_1,G}$ is co-amenable. Moreover, since
$G_1/P\cong\mathbb Z$ is amenable, the inclusion
$i_{P,G_1}$ is co-amenable.
Choose an epimorphism $\psi:P\longrightarrow \mathbb{Z}$ arbitrarily and set $K = \ker(\psi)$. 
The corresponding
infinite cyclic cover of a fiber is a connected non-compact surface, so 
$K$ is a free group. 
Moreover, since $P/K\cong \mathbb{Z}$, $K$ is co-amenable in $P$. 
By Lemmas \ref{l_Lemma3.1} and \ref{l_Lemma6.1}, 
\[
\|i_{K,G}^*[\mathrm{Vol}(i_G)]\|_\infty=\|i_{K,P}^*\circ i_{P,G_1}^*\circ i_{G_1,G}^*[\mathrm{Vol}(i_G)]\|_\infty=\|[\mathrm{Vol}(i_G)]\|_\infty=\boldsymbol{v}_3.
\]

Next we suppose that $N$ has infinite volume.
By \cite[Theorem 1]{so_duke}, if all ends of $N$ is geometrically finite, then 
$[\mathrm{Vol}(i_G)]=0$ in $H_b^3(G,\mathbb{R})$.
Hence (i) holds.
So it remains the case that $N$ has a simply degenerate end $e$.
Let $E$ be the neighborhood of $e$ in $N$ with respect to a fixed finite core $\widehat C$ and let $C$ be a compact core of $N$ contained in $\widehat C$, see Figure \ref{f_core}.
Suppose that $E_C$ is the component of $N\setminus \mathrm{Int}\,C$ containing $E$ and $S_C$ is the frontier of $E_C$ in $N$.
Note that $E_C$ is homeomorphic to $S_C\times [0,\infty)$.
Consider the covering $p:\widetilde N\longrightarrow N$ associated with $\widetilde G:=i_*(\pi_1(S_C))
\subset  \pi_1(N)=G$, where $i:S_C\longrightarrow N$ is the inclusion.
Then $E_C$ is lifted to a submanifold $\widetilde E_C$ of $\widetilde N$, which 
contains a submanifold $\widetilde E$ such that $p|_{\widetilde E}:\widetilde E\to N$ is 
an isometry onto $E$.
It follows that $\widetilde N$ has a simply degenerate end corresponding to $e$ via $p$.
Again by \cite[Theorem 1]{so_duke}, we have 
\begin{equation}\label{eqn_(11)}
\|i_{\widetilde G,G}^*[\mathrm{Vol}(i_G)]\|_\infty=\boldsymbol{v}_3.
\end{equation}
A compact core $\widetilde C$ of $\widetilde N$ is either homeomorphic to $\Sigma_1\times [0,1]$ for some closed connected surface $\Sigma_1$ of genus $>1$ or a compression body 
the boundary $\partial \widetilde C$ of which consists of closed surfaces $\Sigma_0,\Sigma_1,\dots, \Sigma_m$ such that  
$\mathrm{genus}(\Sigma_0)>\mathrm{genus}(\Sigma_i)>0$ for $i=1,\dots,m$.
By Lemma 2.4.2 and Section 2.2 in \cite{bw}, $\widetilde G$ has a
free-product decomposition
\begin{equation}\label{eqn_(12)}
\widetilde G\cong F_r*Q_1*\cdots * Q_m,\quad  Q_i = \pi_1(\Sigma_i),
\end{equation}
for some finitely generated free group $F_r$. 
In the former $I$-product case, we suppose that $F_r=\{1\}$ and $m=1$.
When $m = 0$, take $K=\widetilde G = F_r$.
Then $K$ is free and \eqref{eqn_(11)} proves (ii).

Assume $m\geq 1$. 
For each $i$, choose an epimorphism $\chi_i:Q_i\longrightarrow \mathbb{Z}$ arbitrarily. 
By the fundamental property
of the free product, there is an epimorphism
$\chi:\widetilde G\longrightarrow \mathbb{Z}$ with 
$\chi(F_r)=\{1\}$ and $\chi|_{Q_i}=\chi_i$ for $i=1,\dots,m$.
We will show that $K=\ker(\chi)$ is free. 
The Kurosh subgroup theorem (for example see \cite[Corollary 4.9.1]{mks}) gives sets
of double-coset representatives $T_i\subset \widetilde G$ and a (possibly infinitely generated) free group 
$\mathcal{F}$ such that
factors,
\begin{equation}\label{eqn_(14)}
K\cong \mathcal{F} * \bigl(*_{i=0}^m \bigl(*_{t\in T_i}K\cap tQ_it^{-1}\bigr)\bigr),
\end{equation}
where $Q_0=F_r$.
For $i = 0$, every intersection $K\cap tF_rt^{-1}$ is a subgroup of a free group, so it is free by
Nielsen–Schreier's theorem (for example see \cite[Corollary 2.9]{mks}).
Since $K$ is a normal subgroup of $\widetilde G$, 
\[
K \cap tQ_it^{-1} = t(K\cap Q_i)t^{-1} = t\ker(\chi_i)t^{-1}
\]
for $i\geq 1$. 
Since $\ker(\chi_i)$ is a fundamental group of an infinite cyclic covering of $\Sigma_i$, it is a free group. 
Every factor in \eqref{eqn_(14)} is therefore free.
Since any free product of free groups is free, $K$ is also free. 
Since $\widetilde G/K\cong \mathbb{Z}$,  $K$ is co-amenable in $\widetilde G$. 
Hence Lemma \ref{l_Lemma3.1} and \eqref{eqn_(11)} yield
\[
\|i_{K,G}^*[\mathrm{Vol}(i_G)]\|_\infty=\|i_{K,\widetilde G}^*\circ i_{\widetilde G,G}^*[\mathrm{Vol}(i_G)]\|_\infty
=\|i_{\widetilde G,G}^*[\mathrm{Vol}(i_G)]\|_\infty= \boldsymbol{v}_3.
\]
This proves (ii) and so completes the proof of Lemma \ref{l_Proposition 7.1}.
\end{proof}

\section{Proofs of Theorem \ref{thm_A} and Corollary \ref{thm_B}}

\begin{proof}[Proof of Theorem \ref{thm_A}]
First we consider the case when $G$ is torsion-free, so one can apply 
Lemmas \ref{l_Lemma6.1} and \ref{l_Proposition 7.1}.
Since $G$ is finitely generated and discrete by the assumption of the theorem, 
one of the conditions (i) and (ii) of Lemma \ref{l_Proposition 7.1} holds.
Since $M$ has a simply degenerate end, by \cite[Theorem 1]{so_duke} 
$\|[\mathrm{Vol}(\rho_M)]\|_\infty=\boldsymbol{v}_3$.

Since $\rho^*:H_b^3(G,\mathbb{R})\longrightarrow H_b^3(\varGamma,\mathbb{R})$ is a seminorm non-increasing 
homomorphism, if $\|[\mathrm{Vol}(i_G)]\|_\infty=0$ in $H_b^3(G,\mathbb{R})$, then 
\[
\|[\mathrm{Vol}(\rho)]\|_\infty=\|[\mathrm{Vol}(i_G\circ \rho)]\|_\infty=\|\rho^*[\mathrm{Vol}(i_G)]\|_\infty=0.
\]
Thus we have
\begin{align*}
\|[\mathrm{Vol}(\rho)]-[\mathrm{Vol}(\rho_M)]\|_\infty&\geq \|[\mathrm{Vol}(\rho_M)]\|_\infty-
\|[\mathrm{Vol}(\rho)]\|_\infty\\
&=\|[\mathrm{Vol}(\rho_M)]\|_\infty=\boldsymbol{v}_3.
\end{align*}
This shows \eqref{eqn_geq_v3}.

So it suffices to consider the case of $\|[\mathrm{Vol}(i_G)]\|_\infty\neq 0$.
Lemma \ref{l_Proposition 7.1} gives a free subgroup $K$ of $G$ with
\begin{equation}\label{eqn_(16)}
\|i_{K,G}^*[\mathrm{Vol}(i_G)]\|_\infty = \boldsymbol{v}_{3}.
\end{equation}
By Lemma \ref{l_Lemma4.1}, there is a finitely generated free subgroup $A$ of $ \Gamma$ with $\rho(A) = G$. 
Lemma \ref{l_Lemma5.1} gives 
$
i_{A,\varGamma}^*[\mathrm{Vol}(\rho_M)]= 0
$
in $H_b^3(A,\mathbb{R})$.
Choose a free basis $\mathcal{B}$ of $K$. Since $\rho|_{A} \colon A \rightarrow G$ is surjective, for each $b \in \mathcal{B}$ there exists an element $\widetilde b \in A$ with $\rho(\widetilde{b}) = b$. 
The universal property of the free group defines a homomorphism
$
s \colon K \longrightarrow A
$
with $(\rho|_{A}) \circ s = i_{K,G}$.  
It follows that 
\begin{align*}
\|[\mathrm{Vol}(\rho)]-[\mathrm{Vol}(\rho_M)]\|_\infty&\ge 
\|i_{A,\varGamma}^*([\mathrm{Vol}(\rho)]-[\mathrm{Vol}(\rho_M)])\|_\infty\\
&=\|[\mathrm{Vol}(\rho|_A)]-[i_{A,\varGamma}^*\mathrm{Vol}(\rho_M)]\|_\infty\\
&=\|[\mathrm{Vol}(\rho|_A)]\|_\infty=\|(\rho|_A)^*[\mathrm{Vol}(i_G)]\|_\infty\\
&\ge \|s^*(\rho|_A)^*[\mathrm{Vol}(i_G)]\|_\infty =\|i_{K,G}^*[\mathrm{Vol}(i_G)]\|_\infty=\boldsymbol{v}_3.
\end{align*}
Thus \eqref{eqn_geq_v3} holds if $G$ is torsion-free.

Next we consider the case when $G$ has a torsion.
By Selberg's Lemma (for example see \cite[Theorem 2.29]{mt}), there exists 
a finitely generated torsion-free subgroup $\widetilde G$ of $G$ with finite index.
Then $\widetilde\varGamma=\rho^{-1}(\widetilde G)$ is a subgroup of $\varGamma$ of finite index.
In particular, $\widetilde \varGamma$ is a closed surface group of genus $>1$.
Since moreover $\widetilde M=\mathbb{H}^3/\widetilde \varGamma$ finitely covers $M$, $\widetilde M$ also has 
a simply degenerate end.
Since $1\in \widetilde G$, we have
$\ker(\rho)\subset\widetilde\varGamma$ and hence
$\ker(\rho|_{\widetilde\varGamma})=\ker(\rho)\neq\{1\}$.
So one can apply the former result to  
the torsion-free representations $\rho_M|_{\widetilde \varGamma}$ and $\rho|_{\widetilde \varGamma}$.
Then  
\begin{align*}
\|[\mathrm{Vol}(\rho)]-[\mathrm{Vol}(\rho_M)]\|_\infty
&\geq \|i_{\widetilde \varGamma,\varGamma}^*([\mathrm{Vol}(\rho)]-[\mathrm{Vol}(\rho_M)])\|_\infty\\
&=
 \|[\mathrm{Vol}(\rho|_{\widetilde \varGamma})]-[\mathrm{Vol}(\rho_M|_{\widetilde \varGamma})]\|_\infty
\geq \boldsymbol{v}_3.
\end{align*}
Thus \eqref{eqn_geq_v3} holds in the torsion case as well as in the torsion-free case.
\end{proof}

\begin{proof}[Proof of Corollary \ref{thm_B}]
Theorem \ref{thm_A} implies that, if the inequality \eqref{eqn_geq_v3} fails, then 
$\rho$ is either non-discrete or faithful.
However Claim 1 in the proof of \cite[Theorem D]{so_tran}\footnote{The proof is due to the referee of \cite{so_tran} and not to the author.} shows that 
$\rho$ is discrete in this case and hence $\rho$ must be faithful.
Then $M_\rho=\mathbb{H}^3/\rho(\varGamma)$ is a complete hyperbolic 3-manifold 
with a marking $i_\rho:\Sigma\longrightarrow M_\rho$.
Under the natural identification of $H_b^3(\varGamma,\mathbb{R})=H_b^3(\Sigma,\mathbb{R})$, we have 
$[\mathrm{Vol}(\rho_M)]=i_M^*[\omega_M]$ and $[\mathrm{Vol}(\rho)]=i_\rho^*[\omega_{M_\rho}]$.
Let $\varphi:M\longrightarrow M_\rho$ is a marking-preserving homeomorphism, 
that is, $\varphi\circ i_M$ is homotopic to $i_\rho$.
Since $\varphi\circ i_M$ is homotopic to $i_\rho$, we have
$i_M^*\circ\varphi^*=i_\rho^*$ on bounded cohomology.
Since moreover $i_M^*:H_b^3(M,\mathbb{R})\longrightarrow H_b^3(\Sigma,\mathbb{R})$ is an isometric isomorphism, 
the failure of \eqref{eqn_geq_v3} implies  
\begin{align*}
\|\varphi^*[\omega_{M_\rho}]-[\omega_M]\|_\infty
&=\|i_M^*\circ \varphi^*[\omega_{M_\rho}]-i_M^*[\omega_M]\|_\infty\\
&=\|[\mathrm{Vol}(\rho)]-[\mathrm{Vol}(\rho_M)]\|_\infty<\boldsymbol{v}_3
\end{align*}
in $H_b^3(M,\mathbb{R})$.
So the proof is reduced to the comparison theorem \cite[Theorem A]{so_tran} for bounded fundamental 
classes of complete hyperbolic 3-manifolds, which finishes the proof of Corollary \ref{thm_B}.
\end{proof}

\section*{Acknowledgement}
Y.~Nakano was supported by JSPS KAKENHI Grant  Number 23K03188 and JST PRESTO Grant Number JPMJPR25K8.
T.~Soma was supported by JSPS KAKENHI Grant Number 26K06888.

\end{document}